\documentclass[11pt]{amsart}
\usepackage{fullpage} 
\usepackage{amsmath}

\usepackage{ amsthm,url, 
  amssymb,setspace,epsfig, mathrsfs,fontenc, 
  fullpage, MnSymbol}
\usepackage[alphabetic]{amsrefs}

\newtheorem{thm}{Theorem}
\newtheorem{lem}[thm]{Lemma}

\newtheorem{prop}[thm]{Proposition}

\newtheorem{cor}[thm]{Corollary}

\DefineSimpleKey{bib}{myurl}
\newcommand\myurl[1]{\url{#1}}
\BibSpec{webpage}{
  +{}{\PrintAuthors} {author}
  +{,}{ \textit} {title}
  +{}{ \parenthesize} {date}
  +{,}{ \myurl} {myurl}
}

\newcommand{\nc}{\newcommand}

\nc{\ssec}{\subsection}

\nc{\on}{\operatorname}

\nc {\cG} {\mathcal{G}}
\nc {\cK}{\mathcal{K}}
\nc {\cC} {\mathcal{C}}
\nc {\cL} {\mathcal{L}}
\nc {\cE} {\mathcal{E}}
\nc {\cM} {\mathcal{M}}
\nc {\cO}{\mathcal{O}}
\nc {\cF}{\mathcal{F}}
\nc {\cZ}{\mathcal{Z}}

\nc {\bZ}{\mathbb{Z}}

\nc {\uG} {\underline{G}}
\nc {\cB}{\mathcal{B}}
\nc{\rat}{\mathrm{rat}}
      
\nc {\fk}{\mathfrak{k}}
\nc {\fI}{\mathfrak{i}}
\nc {\fg} {\mathfrak{g}}
\nc {\fu} {\mathfrak{u}}
\nc {\fl} {\mathfrak{l}}
\nc {\fn} {\mathfrak{n}}
\nc {\cP} {\mathcal{P}}
\nc {\fz} {\mathfrak{z}}
\nc {\fc}{\mathfrak{c}}
\nc {\fh}{\mathfrak{h}}
\nc {\fp}{\mathfrak{p}}

\nc{\tg} {\mathtt{g}}
\nc {\hfg} {\widehat{\fg}}
\nc {\hG} {\check{G}}

\nc {\bGm} {\mathbb{G}_m}
\nc{\bC}{\mathbb{C}}
\nc{\bV}{\mathbb{V}}
\nc{\bP}{\mathbb{P}}
\nc{\bA}{\mathbb{A}}

\nc {\quo}{\mathopen{ /\!/}}
\nc {\llp} {\mathopen{ (\!(}}
\nc {\rrp} {\mathopen{ )\!)}}
\nc {\llb} {\mathopen{ [\![}}
\nc {\rrb} {\mathopen{ ]\!]}}

\nc{\bCpp} {\bC \llp t \rrp}
\nc{\Sl}{\mathfrak{sl}}

\nc{\ra}{\rightarrow}

\nc {\tU}{\tilde{U}}
\nc {\tSym}{\widetilde{Sym}}

\nc {\Bun}{\mathrm{Bun}}
\nc {\cA}{\mathcal{A}}
\nc {\Fun}{\mathrm{Fun}}
\nc {\crit}{\mathrm{crit}}
\nc {\Ind}{\mathrm{Ind}}
\nc {\Vac}{\mathrm{Vac}}
\nc {\gr}{\mathrm{gr}}
\nc {\ad}{\mathrm{ad}}
\nc {\Sym}{\mathrm{Sym}}
\nc {\Ram}{\mathrm{Ram}}
\nc {\Op}{\mathrm{Op}}
\nc {\Hitch}{\mathrm{Hitch}}
\nc {\fb}{\mathfrak{b}}
\nc{\cDt}{\mathcal{D}^\times}
\nc {\gl}{\mathfrak{gl}}
\nc {\Sp}{\mathfrak{sp}}
\nc {\So}{\mathfrak{so}}
\nc {\bR}{\mathbb{R}}
\nc {\Hom}{\mathrm{Hom}}
\nc {\Irr}{\mathrm{Irr}}
\nc {\GL}{\mathrm{Gl}}
\nc {\SL}{\mathrm{Sl}}
\nc {\Id}{\mathrm{Id}}
\nc {\Tr}{\mathrm{Tr}}
\nc {\rk}{\mathrm{rank}}
\nc {\Spec}{\mathrm{Spec}}
\nc {\rank}{\mathrm{rank}}
\nc {\Ad}{\mathrm{Ad}}
\nc {\tth}{\mathrm{th}}

\begin{document} 
\title{Stabilisers of eigenvectors of finite reflection groups} 

\begin{abstract}{
Let $x$ be an eigenvector for an element of a finite irreducible reflection group $W$. Let $W_x$ denote the subgroup of $W$ which stabilises $x$. 
We provide an upper bound for the number of roots in the root system of $W_x$ .  This generalises a result of Kostant, who showed that every eigenvector with eigenvalue a primitive $h^\mathrm{th}$ root of unity is regular, where $h$ is the Coxeter number of $W$. We also give a Lie-theoretic interpretation of our result in the study of semisimple conjugacy classes over  Laurent series. In a forthcoming paper, we use this result to establish a geometric analogue of a conjecture of Gross and Reeder.}
\end{abstract} 

\author{Masoud Kamgarpour} 
\address{School of Mathematics and Physics, The University of Queensland} 
\email{masoud@uq.edu.au}


\thanks{The author is supported by an ARC DECRA Fellowship.}

\date{\today}


\maketitle 
\tableofcontents

\section{Introduction} 
\subsection{Recollections on reflection groups} Let $V$ be a finite dimensional real vector space endowed with a  positive definite symmetric bilinear form $(.,.)$. Let $W$ be a finite subgroup of $\mathrm{GL}(V)$ generated by reflections. Thus, $W$ is a finite (real) reflection group. By slightly modifying the standard definition of a root system\footnote{For the standard definition see, for instance, \cite{Bourbaki}. In the modified definition that is used in the context of reflection groups all roots have the same length; see \cite{Humphreys}.}, one can attach a root system $\Phi$ to $W$, cf. \cite{Humphreys}. Henceforth, we assume that $\Phi$ is irreducible. 

 The rank of $W$ is, by definition, the dimension of the vector space $V$. We let $n$ denote this integer.
The ring of invariant polynomials  $\bR[V]^W=\Sym(V^*)^W$ is generated by algebraically independent homogenous polynomials $Q_1, \cdots, Q_n$. Let $d_i$ denote the degree of $Q_i$. These integers are known as the \emph{degrees} of $W$. We arrange these so that 
\[
d_1\leq d_2\leq \cdots \leq d_n.
\]
Note that $Q_i$'s are also algebraically independent generators for the invariant ring $\bC[V_\bC]^W$, where $V_\bC:=V\otimes_\bR \bC$.

The largest  degree $d_n$ is known as the \emph{Coxeter number} of $W$ and is denoted by $h$. 
A \emph{Coxeter element} of $W$ is an element which can be written as a product of all simple reflections, for some choice of a basis of simple roots. All Coxeter elements are conjugate; moreover, their order equals $h$. Furthermore, we have
\[
|\Phi|=nh.
\]
We refer the reader to \cite[V.6]{Bourbaki} or \cite{Humphreys} for these facts.

\subsection{Kostant's theorem on eigenvectors} We are interested in the eigenvectors of elements of $W$.  
Let $w\in W$ and $\lambda \in \bC$. An eigenvector for $w$ with eigenvalue $\lambda$ is a non-zero element $x\in V_\bC$ such that 
\[
w\cdot x=\lambda x.
\]

Recall that an element $x\in V_\bC$ is said to be \emph{regular} if its stabiliser in $W$ is trivial; i.e., $W_x=\{1\}$.

\begin{thm}[Kostant] Let $x\in V_\bC$ be an eigenvector of an element of $W$ with eigenvalue a primitive $h^\mathrm{th}$ root of unity.  Then $x$ is regular. 
\label{t:Kostant}
\end{thm} 

\begin{proof} The statement follows from three fundamental results of Kostant \cite[\S 9]{Kostant}. Since these results are proved in the setting of Weyl groups, we also give references for their generalisation to reflection groups: 
\begin{enumerate}  
\item[(i)] Every eigenvalue $\lambda$ must be a $b^\mathrm{th}$ root of unity, where $b$ is a divisor of one of the degrees of $W$, cf. \cite[Theorem 3.4]{Springer}.

\item[(ii)] If $\lambda$ is a primitive $h^\mathrm{th}$ root of unity, then $w$ is a Coxeter element of $W$, cf.  \cite[\S Theorem 32.C]{Kane}.

\item[(iii)] Every eigenvector of a Coxeter element is \emph{regular}; i.e., its stabiliser in $W$ is trivial. The proof for finite reflection groups is implicit in, e.g., \cite[\S 3.19]{Humphreys} or \cite[\S 29.6]{Kane}. 
 \end{enumerate} 
 \end{proof} 
 
Our main theorem is a generalisation of Kostant's theorem, where $h$ is replaced by an arbitrary positive integer.

\subsection{Main theorem} \label{ss:mainthm}
 
 For every $x\in V_\bC$, let $W_x< W$ denote the stabiliser subgroup; i.e., 

\[
W_x=\{ w\in W \, | \, w\cdot x=x \}.
\]

As we shall see in \S \ref{ss:parabolic}, $W_x$ is a parabolic subgroup of $W$. Now let $\Phi_x\subseteq \Phi$ denote the root system of $W_x$. Set 
\[
N(x):=|\Phi|-|\Phi_x|.
\] 
Equivalently, $N(x)$ is the number of roots in $\Phi$ whose pairing with $x$ is non-zero; i.e., 
\[
N(x)=|\{ \alpha \in \Phi \, | \, (\alpha,x)\neq 0\}|.
\]
Note that $x$ is regular if and only if $N(x)=|\Phi|=nh$. 

The following is our main result: 

\begin{thm}\label{t:main} Let $W$ be a  finite irreducible reflection group of rank $n$. Let $b$ a positive integer and let   $x$ be an eigenvector for an element of $W$ with eigenvalue a primitive $b^\mathrm{th}$ root of unity. Then
\begin{equation} \label{eq:main}
N(x) \geq bn.
\end{equation} 
Moreover,  equality is achieved if and only if $b=h$. 
\end{thm} 
We prove this theorem  by using the classification of irreducible finite reflection groups; see \S \ref{s:proof}.
One can interpret this theorem as providing (yet another) characterisation of the Coxeter number.

\subsection{Semisimple conjugacy classes} 
Let $G$ be a simple simply connected algebraic group over $\bC$. Let $\fg$ denote its Lie algebra. Let $\fh$ be a Cartan subalgebra of $\fg$ and let $W$ denote the Weyl group. 

Let $\cK$ be a field containing $\bC$.  The group $G(\cK)$ acts on $\fg(\cK)$ by the adjoint action. The orbits  of this action are the conjugacy classes of $\fg(\cK)$. The set of conjugacy classes is denoted by $\fg(\cK)/G(\cK)$. Similarly, we have the conjugacy classes $\fg(\overline{\cK})/G(\overline{\cK})$.  There is a canonical map
\[
\fg(\cK)/G(\cK)\ra \fg(\overline{\cK})/G(\overline{\cK}) 
\]
which sends a $\cK$-conjugacy class to the unique $\overline{\cK}$-conjugacy class which contains it. 
It is natural to wonder what the image of this map is. 

Restricting attention to \emph{semisimple} conjugacy classes, the map in question can be written as 
\begin{equation}\label{eq:conjugacy}
\fg(\cK)^{\mathrm{ss}}/G(\cK) \ra \fg(\overline{\cK})^{\mathrm{ss}}/G(\overline{\cK}) = \fh(\overline{\cK})/W.
\end{equation}
Here $\fg(-)^{\mathrm{ss}}\subset \fg(-)$ is the subset of semisimple elements. The  equality in the above line is the well-known bijection between semisimple conjugacy classes in $\fg(\overline{\cK})$ and elements of $\fh(\overline{\cK})/W$, cf. \cite[Theorem 2.2.4]{CM}. Using Theorem \ref{t:main}, we can provide a constraint for the image of the above map, when $\cK=\bCpp$ is the field of Laurent series.

Before stating our result, we recall that
\[
\overline{\cK}=\bigcup_{b\in \bZ_{\geq 1}} \bC \llp t^{\frac{1}{b}} \rrp.
\]
Thus, every element of $\fh(\bar{\cK})$ is in $\fh(\bC \llp t^{\frac{1}{b}} \rrp)$ for some positive integer $b$. 

\begin{thm}\label{t:conjugacy} Suppose $X\in \fh(\overline{\cK} )$ is conjugate to a point of  $\fg(\cK)$.\footnote{Equivalently the class of $X$ in $\fh(\overline{\cK})/W$ is in the image of the map defined in \eqref{eq:conjugacy}.}
Let $x\in \fh\setminus \{0\}$ be the leading term of $X$; i.e.
\[
X=xt^{\frac{a}{b}} + \textrm{higher order terms},\quad \quad \gcd(a,b)=1, \quad b>0.
\]
Then $N(x)\geq bn$. Moreover, equality is achieved if and only if $b$ equals the Coxeter number of $\fg$.
\end{thm}

We will prove this theorem  by  using Theorem \ref{t:main} and a result of Springer on eigenspaces of reflection groups; see \S \ref{s:proof}. In the forthcoming paper \cite{KS}, we  use Theorem \ref{t:conjugacy} to prove a geometric analogue of a conjecture of Gross and Reeder \cite{GR}. We refer the reader to \cite{RY} and \cite{GLRY} for other intriguing connections between Lie theory and the Gross-Reeder story.

 \subsection{Acknowledgement}  I would like to thank Jim Humphreys, Gus Lehrer, Peter McNamara, Chul-hee Lee.,  Daniel Sage, and Ole Warnaar for helpful conversations.


\section{Parabolic subgroups}\label{s:parabolic} 
 In this section, we gather some facts regarding parabolic subgroups of finite reflection groups. These facts  are  used in the proof of the main theorem in the \emph{exceptional} types. The reader interested in the classical types (i.e., types $A$, $B$ and $D$) will only need Lemma \ref{l:parabolic} from this section.

\subsection{Recollections} \label{ss:parabolic} We keep the notation introduced in the previous section.   Let $\Delta\subset \Phi$ be a basis of simple roots. 
For every subset $I\subseteq \Delta$, let $W_I$ denote the subgroup of $W$ generated by $I$. Recall that a subgroup of $W$ is said to be \emph{parabolic} if it is of the form $wW_Iw^{-1}$ for some $w\in W$. 

Parabolic subgroups arise as stabilisers of the action of $W$ on $V$; that is, if $x\in V$ then $W_x$ is a parabolic subgroup, cf. \cite[\S 5.2]{Kane}. Similar result holds in the complex situation: 

\begin{lem}\label{l:parabolic}
 Let $x\in V_\bC$. Then $W_x$ is a parabolic subgroup of $W$. 
\end{lem} 
\begin{proof} Write $x=a+bi$ where $a,b\in V$. The action of $w\in W$ is given by $w \cdot x=w \cdot a+(w \cdot b)i$. It follows that $W_x=W_a\cap W_b$. According to a result of Tits and Solomon (cf. \cite[Lemma 2]{Solomon} or \cite{Qi}), the intersection of two parabolic subgroups of $W$ is again a parabolic subgroup. Thus, $W_x$ is a parabolic subgroup.
 \end{proof} 

 Next, let $V_I$ denote the $\bR$-span of $I$ in $V$. Let $\Phi_I=\Phi\cap V_I$. Then $\Phi_I$ is a root system in $V_I$ with corresponding reflection group $W_I$, cf. \cite[\S 1.10]{Humphreys}. 
 
 Let $U_I$ denote the orthogonal complement of $V_I$ with respect to $(.,.)$. By definition, for every $\alpha\in I$ and $u\in U_I$ we have $(\alpha, u)=0$. Thus, we have 
\begin{equation}\label{eq:ann}
s_\alpha(u)=u, \quad \quad \alpha \in I, \quad u\in U_I,
\end{equation} 
where $s_\alpha$ is the reflection defined by the (simple) root $\alpha$.

 \subsection{Invariant quadratic polynomial} \label{ss:quadratic} 
 Let $W$ be a finite irreducible reflection group of rank $n>2$. Let $x$ be an element of $V_\bC$. In this section, we show that if $Q(x)= 0$, where $Q$ is the unique, up to scalar, invariant quadratic polynomial, then $\rank(W_x)\leq n-2$. 
 
  It is easy to write down a quadratic (homogenous) invariant polynomial. Namely, let  $e_1,\cdots, e_n$ be an orthonormal basis of $V$ with respect to $(.,.)$. 
Let $x\in V$ and write $x=\sum x_i e_i$. 
 Then the polynomial $Q\in \bR[V]$ defined by 
 \[
 Q(x):= \sum_{i=1}^n x_i^2
 \]
   is invariant under the action of the orthogonal group. In particular, it is invariant under the action of $W< O(V)$.  Moreover, every homogenous invariant polynomial of degree $2$ is a scalar multiple of $Q$.
   
   We can also regard $\{e_i\}$ as a $\bC$-basis for $V_\bC$. If $x\in V_\bC$,  we  write $x=\sum x_i e_i$ where  $x_i\in \bC$. The polynomial $Q(x)=\sum x_i^2$ is also  the unique, up to scalar, homogenous element of degree $2$ in the invariant ring $\bC[V_\bC]^W$.

 \begin{lem} \label{l:invariant} Let $A$ be a subspace of $V$ of dimension $n-1$. Let $x$ be a non-zero element in $V_\bC$ satisfying $(x,y)=0$ for all $y\in A_\bC$. Then $Q(x)\neq 0$.
 \end{lem} 
 
 \begin{proof} Let $A^\perp\subset V$ denote the orthogonal complement of $A$ with respect to $(.,.)$. Then $A^\perp$ is a line in $V$. Without the loss of generality, suppose $A^\perp$ is not in the hyperplane defined by setting the last coordinate equal to zero. Then there exists \emph{real} numbers $r_1,\cdots, r_{n-1}, r_n=1$ such that 
 \[
 A^\perp = \{ (a_1,\cdots, a_n) \in V \, | \, a_i=r_i.a_n \} 
 \]
 
By assumption  $x\in (A_\bC)^\perp = A^\perp\otimes_\bR \bC$; thus, we also have that $x_i=r_i x_n$ for $i=1,2, \cdots, n$. 
 Hence, 
  \[
 Q(x)=\sum_{i=1}^{n} x_i^2 =  (1+\sum_{i=1}^{n-1} r_i^2)x_n^2\neq 0.
 \]   
  \end{proof}

\begin{cor} \label{c:parabolic}
Let $x$ be a non-zero element of $V_\bC$ and suppose $Q(x)=0$. Then $\rank(W_x)\leq n-2$ 
\end{cor} 

\begin{proof}
By Lemma \ref{l:parabolic}, $W_x$ is a parabolic subgroup. Since $x$ is non-zero, $W_x$ is a proper subgroup. We may assume that $W_x=W_I$ for a set of simple roots $I$. This means that for all $\alpha\in I$, we have $s_\alpha(x)=x$ which, in turn, implies that $(x,\alpha)=0$. It follows that $(x,y)=0$ for all $y\in V_I$. By the previous lemma, $\dim(V_I)=\rank(W_I) \leq n-2$. 
\end{proof}

\subsection{Eigenspaces} 
Let $\zeta$ be a primitive $b^\tth$ root of unity. For each $w\in W$, let $\Omega(w,\zeta) \subseteq V_\bC$ denote the eigenspace of $w$ with eigenvalue $\zeta$; i.e. 
\[
\Omega(w,\zeta):=\{ x\in V_\bC \, | \, w.x=\zeta.x\}.
\]
Let 
\begin{equation}\label{eq:eigenspace} 
\displaystyle V(b):=\bigcup_{w\in W}  \Omega(w,\zeta).
\end{equation} 
Thus, $V(b)$ is the set of all eigenvectors with eigenvalues a primitive $b^\tth$ root of unity.

The following result is due to Springer; see \cite[\S 3]{Springer}.
\begin{prop} \label{p:eigenspace}
$V(b)=\{x\in V_\bC\, | \, Q_i(x)=0\quad  \textrm{for all $i$ such that $b$ does not divide $d_i$.}\}.$
\end{prop}

\subsection{Parabolic eigenspaces} \label{ss:eigenspace} Let $P=W_I$ be a parabolic subgroup of $W$. Let $w\in P$. Recall that we have a natural subspace $V_I\subseteq V$ on which $P$ acts. 
Let $\Omega_{P}(w,\zeta)\subseteq V_I\otimes_\bR \bC$ denote the eigenspace of $w$ with eigenvalue $\zeta$; i.e., 
\[
\Omega_P(w,\zeta):=\{ x\in V_I\otimes_\bR \bC \, | \, w.x=\zeta.x\}.
\]
 We call this the \emph{parabolic eigenspace} of $w\in W_I$. 
It is clear $\Omega_P(w, \zeta) \subseteq \Omega(w,\zeta)$. In fact, we have equality:

\begin{lem} \label{l:eigenspace} 
 Suppose $\zeta\neq 1$. Then $\Omega_P(w, \zeta) = \Omega(w,\zeta)$. 
\end{lem} 
\begin{proof} Let $x\in \Omega(w,\zeta)$.   Recall that we have the orthogonal decomposition $V=V_I\bigoplus U_I$ and its complexification
\[
V_\bC=V_I\otimes_\bR \bC \bigoplus U_I\otimes_\bR \bC.
\]
 Let us write
\[
x=v+u, \quad \quad v\in V_I\otimes_\bR \bC, \quad u\in U_I\otimes_\bR \bC.
\]
By \eqref{eq:ann}, we have $w.u=u$. 
Thus, 
\[
w.v+u=w.v+w.u= w(v+u)=\zeta(v+u).
\] 
It follows that $u-\zeta. u\in (V_I\cap U_I)\otimes_\bR \bC$; hence, $u=0$. It follows that $x\in V_I\otimes_\bR \bC$.
\end{proof}

\subsection{Non-regular elements} Let $b$ be a positive integer and $\zeta$ a primitive $b^{\tth}$ root of unity. The following result is due to Springer, see \cite[Lemma 4.12]{Springer}.

\begin{lem} 
Let $x\in V(b)$ be  a non-regular non-zero element. Then $x\in \Omega(w, \zeta)$, where $w\in W$ is an element with eigenvalue one. 
\end{lem}

\begin{prop}\label{p:nonregular} Let $b$ be an integer bigger than one and 
 let  $x\in V(b)$ be a non-zero non-regular element. Then, we have: 
 \begin{enumerate} 
 \item[(i)] There exists a parabolic subgroup $P$ and $w\in P$ such that  $x\in \Omega_P(w,\zeta)$. 
 \item[(ii)] $b$ divides a degree of $P$.
 \end{enumerate} 
 \end{prop} 

 \begin{proof} According to the previous lemma, $x\in \Omega(w, \zeta)$, where $w\in W$ is an element with eigenvalue one. Let $z\in V_\bC$ denote the corresponding eigenvector. Let $P=W_z$ denote the isotropy group. By definition, $w\in P$. By Lemma \ref{l:parabolic}, $P$ is a parabolic subgroup and by Lemma \ref{l:eigenspace}, the parabolic eigenspace coincides with the usual one. Thus, $x\in \Omega_P(w,\zeta)$, as required. 
 
 Part (ii) follows from the fact that $\Omega_P(w,\zeta)$ is non-empty if and only if $b$ divides a fundamental degree of $P$. 
\end{proof}


\section{Proofs} \label{s:proof}

\subsection{Theorem \ref{t:main} in type $A_{n-1}$} 
 Let $n$ be an integer greater than one. 
Let $V$ be the subspace of the $n$-dimensional affine space consisting of $n$-tuples $(x_1,\cdots, x_n)$ satisfying $\sum x_i=0$. We use the elementary symmetric functions as invariant polynomials. Then $W\simeq S_n$.

\subsubsection{} Let $x\in V_\bC$ be a non-zero element. According to Lemma \ref{l:parabolic}, $W_x$ is a proper parabolic subgroup. It is easy to check that the parabolic of $W$ with the largest number of root is isomorphic to $S_{n-1}$. Thus, for all non-zero $x\in V_\bC$, we have 
\[
|\Phi|-|\Phi_x|\geq n(n-1)-(n-1)(n-2) = 2(n-1)>(n-1).
\]
Therefore, the theorem is evident for $b=1$. 
 Henceforth, we assume $1<b<n$.

 \subsubsection{} 
 Suppose $x=(x_1,\cdots, x_n)\in V(b)$. Then there exists $ 1\leq k\leq \frac{n}{b}$ and complex numbers $a_i$ such that the $x_i$'s are the roots of the polynomial 
\[
P(X):=X^n+a_1 X^{n-b} + \cdots + a_k X^{n-bk}, \quad \quad a_k\neq 0. 
\]
Now let $\alpha_1,\cdots, \alpha_k$ denote the roots of the polynomial 
\[
Q(Y):=Y^k + a_1Y^{k-1} + \cdots + a_k.
\]
Since $a_k\neq 0$ we have that  $\alpha_i\neq 0$ for all $i$. Let $\zeta$ be a primitive $b$th root of unity. Then the roots of $P(X)$ equal 
\[
\zeta^i \alpha_j, \quad \quad i\in \{1,2,\cdots, b\},\quad j\in \{1,2,\cdots, k\}
\]
together with $n-kb$ copies of $0$.

\subsubsection{} Let us ignore the requirement  $\sum x_i=0$. Then the  largest possible stabiliser for $x$  is  achieved  if and only if $\zeta^{i_1} \alpha_1=\zeta^{i_2} \alpha_2 = \cdots = \zeta^{i_{k}} \alpha_k$ for some integers $i_1,\cdots, i_k$.  In this case,
\[
W_x \simeq (S_k)^{b} \times S_{n-bk}.
\]

Thus, for every non-zero $x\in V_\bC$, we have 
\[
\begin{aligned} 
N(x) = |\Phi| - |\Phi_x| & \geq  |\Phi_{S_{n}}| - |\Phi_{(S_k)^{b} \times S_{n-bk}}| \\[1ex]
 			&  =     n(n-1)-[bk(k-1)+(n-kb)(n-kb-1)]\\[1ex]
			& =      2kbn-k^2b^2 - bk^2.
\end{aligned} 
\]

\subsubsection{} We claim that the above quantity is bigger than $b(n-1)$. Indeed, 
if $k=1$, then 
\[
N(x) \geq 2bn-b^2 -b > b(n-1),  
\]
where the last inequality follows from $b<n$. On the other hand, if $k>1$, then 
\[
N(x)= 2kbn-k^2b^2 - bk^2 \geq b(2kn - kn - k^2) > b(n-1). 
\]
Here, the second to last inequality follows from the fact that $kb\leq n$. The last inequality is equivalent to  
\[
n(k-1)>(k-1)(k+2)\iff n>k+2
\]
 which is true because $1<k\leq \frac{n}{2}$. This completes the proof in type $A$. \qed

\subsection{Theorem \ref{t:main} in type $B_n$} As we saw above, the proof in type $A$ has four steps. We sketch how these steps need to be modified in type $B$. Firstly, one checks that the parabolic with the largest number of roots is isomorphic to $B_{n-1}$. Thus, the theorem is evident for $b\leq 3$. Hence, we can assume $3<b<2n$. For the second step, one needs to consider the polynomial
\[
X^{2n}+a_1 X^{2n-b} + \cdots + a_k X^{2n-bk}, \quad \quad a_k\neq 0. 
\]
For the third step, one checks that the largest possible stabiliser is 
\[
W_x\simeq (S_k)^b \times W_{B_{n-bk-1}}.
\] 
Thus, 
\[
N(x)=|\Phi| - |\Phi_x|  \geq |\Phi_{B_{n}}| - |\Phi_{(A_{k-1})^{b} \times \Phi_{B_{n-bk-1}}}| \geq  5bkn - bk^2 - 2b^2k^2-2+2(n-bk).
\]
Finally, it is easy to check that the above quantity is bigger than $bn$. \qed

\subsection{Theorem \ref{t:main} in type $D_n$} 
Note $D_3\simeq A_3$, so we may assume $n\geq 4$. 
It is easy to check that the parabolic subgroup of $D_n$ with the largest number of roots is isomorphic to $D_{n-1}$. Thus, the theorem is evident for $b\leq 3$. Henceforth, we assume $4<b<2n$ and proceed as in type $A$. The polynomial we need to consider is
\[
X^{2n}+a_1 X^{2n-b} + \cdots + a_k X^{2n-bk}, \quad \quad a_k\neq 0, \quad a_i\in \bC.
\]
One checks that the stabiliser with the maximum number of roots is
\[
W_x\simeq (S_k)^b \times W_{D_{n-bk-1}}.
\] 
Therefore, we have 
\[
N(x)=|\Phi| - |\Phi_x|  \geq  |\Phi_{D_{n}}| - |\Phi_{(A_{k-1})^{b} \times \Phi_{D_{n-bk-1}}}| =  2n(n-1)  - [bk(k-1) - 2(n-bk-1)(n-bk-2)].
\]
Finally, one can show that the above quantity is bigger than  $bn$.\qed

\subsection{Theorem \ref{t:main} in exceptional types} We provide the complete proof in the case of $E_6$. The proof in other exceptional types is similar, but easier. In more details, as we shall see, the proof in type $E_6$ has four steps. For types $E_7$ and $E_8$, one only needs to imitate  the first three step.  For type $F_4$ one only needs the first two steps. Finally, for types $G$, $H$ and $I_2$, one only needs to imitate the first step. 

 \subsubsection{} Recall that $E_6$ has $72$ roots and its fundamental degrees are $2, 5, 6, 8, 9, 12$.
It is easy to check that the proper parabolic of $E_6$ with the largest number of roots is $D_5$ with $40$ roots. Thus, for all non-zero $x\in V_\bC$, we have 
\[
N(x)\geq 72-40=32>5\times6.
\]
Hence, the  theorem is evident for $b\leq 5$.

\subsubsection{} Henceforth, we assume $b>5$. 
Let $Q$ denote the unique, up to scalar, homogenous quadratic invariant polynomial; see \S \ref{ss:quadratic}. If $x$ is an eigenvector with eigenvalue a $b^\mathrm{th}$ root of unity, then $Q(x)=0$. By Corollary \ref{c:parabolic}, the stabiliser $W_x$ must be a parabolic subgroup of rank $\leq 4$. The maximum number of roots in a parabolic subgroup of $E_6$ of rank $\leq 4$  is $24$ (for $D_4$). Thus, 
\begin{equation} \label{eq:E_6}
N(x)\geq 72-24=48>6\times 6.
\end{equation} 
Hence, our result also holds for $b=6$. 

\subsubsection{} The remaining cases are $b=8$ and $b=9$. Let $x\in V(9)$ be a non-zero element. According to Proposition \ref{p:nonregular}, if $x$ is not regular, then there must exists a proper parabolic subgroup of $W$ with a fundamental degree divisible by $9$. But there is no such parabolic subgroup of $E_6$. Thus, $x$ must be regular, and so the theorem is immediate. 

\subsubsection{} It remains to treat the case $b=8$. Suppose $x$ is a non-regular non-zero element of $V(8)$. According to Proposition \ref{p:nonregular}, there exists a parabolic subgroup $P=W_I< W$ and $w\in P$ such that 
\[
x\in \Omega(w,\zeta) = \Omega_P(w,\zeta)\subseteq V_I,
\]
where $\zeta$ is a primitive $8^\mathrm{th}$ root of unity. 
 Moreover the parabolic $P$ must have a fundamental degree divisible by $8$. The only possibility is $P\simeq D_5$.  But in this case, $8$ is the \emph{highest} fundamental degree of $P$ and so, by Kostant's theorem (Theorem \ref{t:Kostant}), $x\in V_I$ is \emph{regular} for the reflection action $P$ on $V_I$. In other words, 
 \[
 (\alpha,x)\neq 0,\quad \quad \forall \, \alpha\in I.
 \]
 Since $|I|=5$, it follows that  $(\alpha,x)$ is zero for at most one simple root of $W$. Hence, either $x$ is regular or $W_x\simeq A_1$. In both cases, we have $N(x)>6\times 8$. \qed

\subsection{Theorem \ref{t:conjugacy}}\label{s:conjugacy}
In view of Theorem \ref{t:main}, it is sufficient to show that  $x\in V(b)$; i.e., $x$ is an eigenvector for some element of $W$ with eigenvalue a $b^\mathrm{th}$ root of unity. 

 Let $Q\in \bC[\fg]^G$ be an invariant homogenous polynomial of degree $d$. Note that $Q$ is also an invariant polynomial on $\fg(\cK)$ and $\fg(\overline{\cK})$.   Let $Y\in G(\overline{\cK}).X \cap \fg(\cK)$. Since $X$ and $Y$ are conjugate, we have  $Q(X)=Q(Y)$. On the other hand, $Y$ is a $\cK$-rational point; thus, $Q(X)\in \cK$.
 
  Now observe that
  \[
  Q(X)=Q(x t^{\frac{a}{b}}) + \textrm{higher order terms} = t^{\frac{da}{b}} Q(x) + \textrm{higher order terms}.
  \]
  By assumption $\gcd(a,b)=1$; thus, for the leading term of the above expression to be in $\cK=\bCpp$, we must have 
  \begin{equation} \label{eq:chevalley} 
  \textrm{$Q(x)=0,\quad $ whenever $b$ does not divide $d$}.
  \end{equation} 
   By a theorem of Chevalley, the restriction map provides an isomorphism $\bC[\fg]^G \simeq \bC[\fh]^W$. Thus, \eqref{eq:chevalley} is also satisfied for all homogenous polynomials $Q\in \bC[\fh]^W$ of degree $d$. By Proposition \ref{p:eigenspace}, $x\in V(b)\qedhere$. \qed

  \end{document}